\documentclass[11pt,a4paper]{amsproc}
\usepackage{a4wide}
\usepackage{setspace}
\usepackage[english]{babel}
\usepackage{amsmath}
\usepackage{amssymb}
\usepackage{amsthm}
\usepackage{mathtools}
\usepackage{thmtools}
\usepackage{thm-restate}
\usepackage{float}
\usepackage{subcaption}
\usepackage{xcolor}
\usepackage{url}
\usepackage{graphicx}
\usepackage{mathdots}
\usepackage{pdfpages}
\usepackage{enumitem}

\usepackage{tikz}
\usetikzlibrary{positioning}
\usetikzlibrary{arrows}

\usepackage{amsthm}
\usepackage{hyperref}
\usepackage{cleveref}

\definecolor{celestialblue}{rgb}{0.29, 0.59, 0.82}
\hypersetup{
  linkcolor  = celestialblue,
  colorlinks = true
}

\newcounter{dummy}\numberwithin{dummy}{section}

\theoremstyle{definition}

\theoremstyle{plain}
\newtheorem{Theorem}[dummy]{Theorem}
\newtheorem{MainTheorem}{Theorem}

\newtheorem{Proposition}[dummy]{Proposition}
\newtheorem{Lemma}[dummy]{Lemma}

\theoremstyle{remark}

\crefname{Definition}{definition}{definitions}
\Crefname{Definition}{Definition}{Definitions}

\crefname{Theorem}{theorem}{theorems}
\Crefname{Theorem}{Theorem}{Theorems}
\crefname{MainTheorem}{theorem}{theorems}
\Crefname{MainTheorem}{Theorem}{Theorems}
\crefname{Corollary}{corollary}{corollaries}
\Crefname{Corollary}{Corollary}{Corollaries}
\crefname{Proposition}{proposition}{propositions}
\Crefname{Proposition}{Proposition}{Propositions}
\crefname{Lemma}{lemma}{lemmas}
\Crefname{Lemma}{Lemma}{Lemmas}

\crefname{Remark}{remark}{remarks}
\Crefname{Remark}{Remark}{Remarks}
\crefname{Example}{example}{examples}
\Crefname{Example}{Example}{Examples}
\crefname{Question}{question}{questions}
\Crefname{Question}{Question}{Questions}

\newcommand{\sete}[1]{\left\lbrace #1 \right\rbrace}
\newcommand{\setp}[2]{\left\lbrace #1 \; \middle| \; #2 \right\rbrace}
\newcommand{\compl}[1]{\overline{ #1 }}
\newcommand{\cl}[1]{\overline{ #1 }}

\newcommand{\abs}[1]{\left| #1 \right|}

\DeclareMathOperator{\Int}{Int}
\DeclareMathOperator{\SIN}{SIN}
\DeclareMathOperator{\eps}{\varepsilon}\DeclareMathOperator{\Id}{Id}

\title[Invariant automatic continuity]{Invariant automatic continuity for compact connected simple Lie groups}

\author{Laurens Diels}
\address{L.D., Analysis Section, KU Leuven, 3001 Leuven, Belgium}
\email{laurens.diels@kuleuven.be}

\author{Philip A. Dowerk}
\address{P.A.D., Institut f\"ur Geometrie, TU Dresden, 01062 Dresden, Germany}
\email{philip.dowerk@tu-dresden.de}

\begin{document}
\onehalfspace
	
\maketitle
\begin{abstract}
	In this note we prove that a homomorphism from a compact connected simple Lie group with the norm topology to any separable SIN group is automatically continuous. This generalizes a result by Dowerk and Thom. 
	Further, we prove some elementary characterizations of invariant automatic continuity. 
\end{abstract}

\section{Introduction}
Automatic continuity of group homomorphisms is a classical subject which dates back to a question of Cauchy on the continuity of endomorphisms of $(\mathbb{R},+)$, endowed with the usual topology. See Rosendal's survey \cite{RosendalSurvey} for more historical background on the subject. Recall that a topological group $G$ has \textit{automatic continuity} if every homomorphism from $G$ to any separable topological group is continuous. 
It follows from the work of Kallman \cite{Kallman}, see also \cite[Example 1.5]{RosendalSurvey}, that non-trivial compact connected Lie groups do not have automatic continuity, and in fact it is not known whether there exists a compact or even a locally compact group with automatic continuity.

In contrast, amongst other results, the second named author and Thom have proven in \cite{DowerkThom} that every homomorphism from a 
\begin{itemize}
	\item finite-dimensional projective unitary group $\mathrm{PU}(n),\ n\in\mathbb{N}$,
	\item finite-dimensional special unitary group $\mathrm{SU}(n),\ n\in\mathbb{N}$, 
\end{itemize}
endowed with the norm topology, to any separable SIN group is continuous. Following \cite{DowerkThom}, we say that these compact connected Lie groups have the \textit{invariant automatic continuity} property, abbreviated \textit{(IAC)}. On the other hand, the unitary group $\mathrm{U}(n)$ does have discontinuous homomorphisms into the separable SIN group $\mathbb{R}$, as one can show using a composition of the continuous determinant $\mathrm{U}(n)\to \mathrm{U}(1)$ with a discontinuous homomorphism $\mathrm{U}(1)\to\mathbb{R}$ (which can be constructed using the axiom of choice). 
The notion of invariant automatic continuity has also proven to be fruitful in the context of some compact totally disconnected groups (i.e.\ profinite groups), see the work of Le Ma\^itre and Wesolek \cite{LeMaitreWesolek}.

The main goal of this short article is to generalize the (IAC) results on $\mathrm{PU}(n)$ and $\mathrm{SU}(n)$ to compact connected simple Lie groups, which is proven in Section \ref{sec_cpt}, see Theorem \ref{th:IACcptsimple}. Here \textit{simple} means that the Lie group is non-abelian and has no non-trivial closed connected normal subgroups. 
\begin{MainTheorem}
	Every homomorphism from a compact connected simple Lie group to any separable SIN group is continuous.
\end{MainTheorem}

Further we obtain a few characterizations of invariant automatic continuity in Section 2, inspired by the work of Rosendal \cite{Rosendal} and Rosendal and Solecki \cite{RosendalSolecki} on automatic continuity. We do not assume topological groups to be Hausdorff.
\begin{MainTheorem}\label{th:main2}
	Let $(G,\tau)$ be a topological group. The following are equivalent.
\begin{enumerate}[label=(\roman*)]
	\item The topological group $G$ has the invariant automatic continuity property.
	\item Every homomorphism from $G$ to a $\aleph_0$-bounded SIN group is continuous.
	\item Every homomorphism from $G$ to a Polish SIN group is continuous.
	\item\label{it:universal} Every homomorphism from $G$ to a fixed universal Polish SIN group is continuous. 
	\item $\tau'\subseteq \tau$ for every $\aleph_0$-bounded SIN group topology $\tau'$ on $G$.
	\item $\tau'\subseteq \tau$ for every second countable SIN group topology $\tau'$ on $G$.
\end{enumerate}  
\end{MainTheorem}
Except for equivalence \ref{it:universal} in \Cref{th:main2} (which is also \Cref{prop:Universal}), all results are taken from the master's thesis of the first named author under supervision of the second named author.

\section{Characterizations of invariant automatic continuity}
\subsection{Preliminaries on topological groups and invariant automatic continuity}
A \textit{topological group} $G$ is a group with a topology such that the operations of multiplication and taking inverses are continuous. We do not require the group topology to be Hausdorff. This is the case precisely when the singleton $\sete{1_G}$ of the neutral element is closed; $\cl{\sete{1_G}}$ is always a normal subgroup. Topological groups are always regular \cite[Theorem 1.3.14]{ArhangelskiiTkachenko}, i.e.\ if $G$ is a topological group, then for every $g\in G$ and for every open neighborhood $U$ of $g$ there exists an open neighborhood $V$ of $g$ such that $\overline{V}\subseteq U$.

When using the terminology \textit{homomorphism} we will always mean algebraic homomorphism, i.e.\ without continuity imposed. Since the left and right translation maps are homeomorphisms, it follows that a homomorphism between topological groups is continuous if and only if it is continuous at the neutral element. Bijective bicontinuous homomorphisms are called \textit{topological isomorphisms}. A \textit{topological embedding} is a map $\varphi: G \to H$ between topological groups such that its surjective restriction $\varphi: G \to \varphi(G)$ is a topological isomorphism. 

The following lemma follows from the continuity of the $k$-ary multiplication map. 
\begin{Lemma}\label{Lemma}
Let $G$ be a topological group and $U\subseteq G$ an open neighborhood of $1_G$. Then for any $k \in \mathbb{N} \setminus \sete{0}$ there exists an open neighborhood $V \subseteq G$ of $1_G$ such that $1_G\in V\subseteq V^k\subseteq U$. 
\end{Lemma}

The basic theorem for metrizability of topological groups is by Birkhoff and Kakutani, see \cite{Birkhoff} and \cite{Kakutani}. 
\begin{Theorem}[Birkhoff-Kakutani]
\label{th:BirkhoffKakutani}
A topological group is metrizable if and only if it is Hausdorff and first countable. If this is the case, then it admits a compatible left-invariant metric.
\end{Theorem}

Given a homomorphism $\varphi: G \to H$ between topological groups one can obtain potentially new group topologies: the initial and the final topology. 
If $N\trianglelefteq G$, then the final topology on $G/N$ for the natural projection homomorphism $\pi: G \to G/N: g \mapsto gN$ is called the \textit{quotient topology}. With respect to this topology, $\pi$ is both continuous (by definition) and open. The quotient topology is Hausdorff if and only if $N$ is closed in $G$. 

In order to show that we may assume the target groups in the definition of (IAC) to be Hausdorff, we will use the following fact.
\begin{Proposition}
\label{th:continuousTargetQuotientByClosureOfIdentity}
Let $\varphi: G \to H$ be a group homomorphism between topological groups. Consider $\widetilde{\varphi}: G \to H/\cl{\sete{1_H}}: g \mapsto \varphi(g)\cl{\sete{1_H}}$. Then $\varphi$ is continuous if and only if $\widetilde{\varphi}$ is so.
\end{Proposition}
\begin{proof}
Denote the natural projection map $H \to H/\cl{\sete{1_H}}$ by $\pi$. Then $\widetilde{\varphi} = \pi \circ \varphi$, so that $\widetilde{\varphi}$ is continuous if $\varphi$ is so (since $\pi$ is continuous by definition of the quotient topology).

Conversely, assume that $\widetilde{\varphi}$ is continuous. It suffices to show that $\varphi$ is continuous at $1_G$. Let $V \subseteq H$ be an open neighborhood of $1_H = \varphi(1_G)$. Use \Cref{Lemma} to find an open neighborhood $W_1 \subseteq H$ of $1_H$ such that $W_1 \subseteq W_1^2 \subseteq V$. As any topological group is regular, we can find an open neighborhood $W_2$ of $1_H$ such that $\cl{W_2} \subseteq W_1$. So
\begin{equation}
\label{eq:closureSubsetSquareSubset}
1_H \in W_2 \subseteq \cl{W_2} \subseteq W_1 \subseteq W_1^2 \subseteq V.
\end{equation} Since $\pi$ is open, $\pi(W_2) \subseteq H/\cl{\sete{1_H}}$ is an open neighborhood of $\pi(1_H) = \widetilde{\varphi}(1_G)$. By continuity of $\widetilde{\varphi}$ at $1_G$ there exists an open neighborhood $U \subseteq G$ of $1_G$ such that $\widetilde{\varphi}(U) \subseteq \pi(W_2)$, i.e.\ $\pi(\varphi(U)) \subseteq \pi(W_2)$ and $\varphi(U) \subseteq \pi^{-1}(\pi(W_2))$. Note that $\pi^{-1}(\pi(W_2)) = W_2 \cl{\sete{1_H}}$. Using that $\cl{\sete{1_H}} \subseteq \cl{W_2}$ since $1_H \in W_2$, we obtain that
\[\varphi(U) \subseteq W_2\cl{W_2} \subseteq W_1^2 \subseteq V\]
by \eqref{eq:closureSubsetSquareSubset}. We conclude that $\varphi$ is continuous at $1_G$ and thus continuous everywhere.
\end{proof}

A subset $A \subseteq G$ is called \textit{countably syndetic} if countably many translates of $A$ cover $G$, that is, if there exists a countable subset $T \subseteq G$ such that
\[G = TA = \setp{ta}{t \in T, a \in A}.\] 
If every non-empty open subset of $G$ is countably syndetic, then $G$ is called \textit{$\aleph_0$-bounded} or \textit{$\omega$-narrow}. One can show that every separable group is $\aleph_0$-bounded and that $\aleph_0$-boundedness passes to subgroups, products and quotients \cite[Propositions 3.4.2 and 3.4.3, Theorem 3.4.4 and Corollary 3.4.8]{ArhangelskiiTkachenko}. 
One can characterize Hausdorff $\aleph_0$-bounded groups as follows, see e.g. \cite[Theorem 3.4.23]{ArhangelskiiTkachenko} and its proof. 
\begin{Theorem}[Guran]
\label{th:Guran}
A Hausdorff group is $\aleph_0$-bounded if and only if it is topologically isomorphic to a subgroup of a product of second countable groups. In addition, we can take these second countable groups to be metrizable.
\end{Theorem}

We say that a subset $A \subseteq G$ is \textit{conjugacy-invariant} if $A = gAg^{-1} = \setp{gag^{-1}}{a \in A}$ for all $g \in G$.
A topological group $G$ is called \textit{SIN} (short for Small Invariant Neighborhoods) or \textit{balanced} if it admits a local basis at $1_G$ consisting of conjugacy-invariant open neighborhoods. As the natural projection homomorphism is open, SIN passes to quotients. Furthermore, being SIN is preserved under the initial topology. In particular, products and subgroups of SIN groups remain SIN. Every compact group is SIN \cite[Corollary 1.8.7]{ArhangelskiiTkachenko}.

There is a SIN version of the Birkhoff-Kakutani Theorem, see \cite[Corollary 3.3.14]{ArhangelskiiTkachenko}.
\begin{Theorem}
\label{th:BirkhoffKakutaniSIN}
A metrizable group is SIN if and only if it admits a compatible bi-invariant metric.
\end{Theorem}

One can adapt the proof of Guran's Theorem \cite{Guran} in \cite{ArhangelskiiTkachenko} in the context of SIN groups. To this end we need a SIN version of \Cref{Lemma}. Recall that a subset $A \subseteq G$ of a group $G$ is called \textit{symmetric} if $A = A^{-1} = \setp{a^{-1}}{a \in A}$.
\begin{Lemma}
\label{LemmaSIN}
Let $G$ be a SIN group and $U \subseteq G$ an open neighborhood of $1_G$. Then for any $k \in \mathbb{N} \setminus \sete{0}$ there exists a symmetric conjugacy-invariant open neighborhood $V \subseteq G$ of $1_G$ such that $1_G \in V \subseteq V^k \subseteq U$.
\end{Lemma}

\begin{Theorem}
\label{th:GuranSIN}
A Hausdorff group is $\aleph_0$-bounded and SIN if and only if it is topologically isomorphic to a subgroup of a product of second countable SIN groups. In addition, we can take these second countable SIN groups to be metrizable.
\end{Theorem}
\begin{proof}
Since both being SIN and $\aleph_0$-boundedness pass to products and subgroups, we only need to show that $\aleph_0$-bounded Hausdorff SIN groups can be topologically embedded into a product of second countable metrizable SIN groups. This will be achieved using the following more general construction, see also \cite[Theorem 3.4.21]{ArhangelskiiTkachenko}.

If $\mathcal{P}$ is a property of topological groups, then a topological group $G$ is called \textit{range-$\mathcal{P}$} if for every open neighborhood $U \subseteq G$ of $1_G$ there exist a $\mathcal{P}$-group $H_U$, an open neighborhood $V_U \subseteq H_U$ of $1_{H_U}$ and a continuous homomorphism $p_U: G \to H_U$ such that $p_U^{-1}(V_U) \subseteq U$. If $G$ is also Hausdorff, then one can show that the map
\[\Delta: G \to \prod_{U} H_U: g \mapsto (p_U(g))_U,\]
where the product is over all open neighborhoods $U \subseteq G$ of $1_G$, is a topological embedding into a product of $\mathcal{P}$-groups. Therefore, it suffices to show that Hausdorff $\aleph_0$-bounded groups are range-(second countable metrizable SIN).

Now let $G$ be an $\aleph_0$-bounded Hausdorff SIN group and fix an open neighborhood $U \subseteq G$ of $1_G$. Since $G$ is SIN we can find a symmetric conjugacy-invariant open neighborhood $U_0 \subseteq G$ of $1_G$ with $U_0 \subseteq U$. Apply \Cref{LemmaSIN} to inductively obtain a sequence $(U_n)_{n \in \mathbb{N}}$ of symmetric conjugacy-invariant open neighborhoods of $1_G$ satisfying $U_{n + 1}^2 \subseteq U_n$ for all $n \in \mathbb{N}$.

Denote the set of dyadic rationals by $\mathbb{D}$. For every $r = \frac{m}{2^n} \in \mathbb{D} \cap (0, 1]$, where $m, n \in \mathbb{N}$, we define open neighborhoods $V(r)$ of $1_G$ by induction on $n$ as follows. We put $V(1) = U_0$ and having defined $V(m/2^n)$ for a certain $n \in \mathbb{N}$ and all $1 \leq m \leq 2^n$, we put $V(2^{-n - 1}) = U_{n + 1}$ and $V(\frac{2m+1}{2^{n + 1}}) = V(2^{-n-1})V(m/2^n)$ for $m \in \sete{1, \ldots, 2^n - 1}$ (note that $V(\frac{2m}{2^{n + 1}}) = V(m/2^n)$ is already defined). We also choose $V(r) = G$ for dyadic rationals $r > 1$. We use the neighborhoods $V(r)$ for $r \in \mathbb{D}_{>0}$ to construct maps
\[f: G \to [0, 1]: g \mapsto \inf \setp{r \in \mathbb{D}_{>0}}{g \in V(r)}\]
and
\[\ell: G \to \mathbb{R}_{\geq 0}: g \mapsto \sup_{h \in G} \abs{f(gh) - f(h)}.\]

One can check
\begin{itemize}
\item that
\[Z \coloneqq \setp{g \in G}{\ell(g) = 0} = \bigcap_{n \in \mathbb{N}} U_n\]
is a normal subgroup of $G$,
\item that
 \[d: Q \times Q \to \mathbb{R}_{\geq 0}: (gZ, hZ) \mapsto \ell(g^{-1}h),\]
where $Q \coloneqq G/Z$, is a well-defined left-invariant metric on $Q$ which induces a second countable group topology on $Q$,
\item and that the natural projection homomorphism $\pi: G \to Q: g \mapsto gZ$ is continuous with respect to $d$ and satisfies $\pi^{-1}(B_d(1_Q, 1)) \subseteq U$, where $B_d(1_Q, 1)$ is the open $d$-ball in $Q$ with center $1_Q$ and radius $1$,
\end{itemize}
see \cite[pp.~165-168]{ArhangelskiiTkachenko}. This shows that $G$ is range-(second countable metrizable).

It remains to show that $Q$ is SIN, for which it suffices by \Cref{th:BirkhoffKakutaniSIN} to show that $d$ is right-invariant (hence bi-invariant). Since $U_n$ is conjugacy-invariant for all $n \in \mathbb{N}$ and the product of conjugacy-invariant sets remains conjugacy-invariant, it follows from the inductive construction that $V(r)$ is conjugacy-invariant for all $r \in \mathbb{D}_{>0}$. Therefore $f$ satisfies $f(ghg^{-1}) = f(h)$ for all $g, h \in G$, which in turn implies the same for $\ell$. This readily leads to the right-invariance of $d$.
\end{proof}

Recall that a topological group is \textit{Polish} if it is separable and admits a compatible complete metric. Polish groups satisfy an open mapping theorem, see \cite[Corollary 4.3.34]{ArhangelskiiTkachenko}: Every surjective continuous homomorphism between Polish groups is open. 

\begin{Theorem}
\label{th:separableMetrisableEmbedsPolish}
Let $G$ be a separable metrizable group. Then there exists a Polish group $\compl{G}$ in which $G$ can be topologically embedded as a dense subgroup. Moreover, if $G$ is SIN, then we can also take $\compl{G}$ to be SIN.
\end{Theorem}
\begin{proof}
For a separable metrizable group $G$ this follows from \cite[Theorem 2.1.3]{Gao} in combination with the Birkhoff-Kakutani Theorem \ref{th:BirkhoffKakutani}. If in addition $G$ is SIN then instead of \Cref{th:BirkhoffKakutani} we use \Cref{th:BirkhoffKakutaniSIN}.
\end{proof}

Following \cite{DowerkThom} we say that a topological group $G$ has the \textit{invariant automatic continuity property}, or in short \textit{(IAC)}, if every group homomorphism $G \to H$ to any separable SIN group $H$ is continuous. We further need the adapted notion of Steinhaus groups.
A topological group $G$ is \textit{invariant Steinhaus} if there exists a $k \in \mathbb{N}$ such that $1_G \in \Int W^k$ for every symmetric conjugacy-invariant countably syndetic $W \subseteq G$.
A variant of \cite[Proposition 2]{RosendalSolecki} in the SIN setting is \cite[Proposition 8.10]{DowerkThom}:

\begin{Theorem}[Dowerk-Thom]\label{th:invSteinhaus}
Every invariant Steinhaus group has the invariant automatic continuity property. 
\end{Theorem}

\Cref{th:invSteinhaus} was used in \cite{DowerkThom} to show that the compact connected simple Lie groups $\mathrm{PU}(n)$ and $\mathrm{SU}(n)$ have the invariant automatic continuity property when endowed with the usual norm topology. These groups are known not to have the automatic continuity property, see \cite[Example 1.5]{RosendalSurvey}. Additionally, \Cref{th:invSteinhaus} was used in \cite{DowerkThom} to show that the (non-locally compact, Polish) projective unitary group of a separable type $\mathrm{II}_1$ factor von Neumann algebra has the invariant automatic continuity property with respect to the strong operator topology. The notion of invariant automatic continuity can be used to show the uniqueness of Polish SIN topologies. Indeed, if $(G, \tau)$ is a Polish SIN group with (IAC), then the open mapping theorem for Polish groups implies that $\tau$ is the only Polish SIN topology $G$ admits. Invariant automatic continuity has also proven to be fruitful in the context of profinite groups, using again Theorem \ref{th:invSteinhaus}, see \cite{LeMaitreWesolek}.

\subsection{Characterizations of invariant automatic continuity}
We present an elementary characterization of invariant automatic continuity, which allows us to consider only homomorphisms into Polish SIN groups instead of potentially non-Hausdorff separable SIN groups. 
The following \namecref{th:IACTarget} is motivated by the paragraph below \cite[Proposition 2]{RosendalSolecki} and analogously holds for automatic continuity instead of invariant automatic continuity, i.e.\ after removing SIN from the conditions. 
\begin{Theorem}
\label{th:IACTarget}
Let $G$ be a topological group. The following are equivalent.
\begin{enumerate}[label=(\roman*)]
\item \label{it:target:IAC} The topological group $G$ has (IAC).
\item \label{it:target:IACAleph0} Every homomorphism $G \to H$ to an $\aleph_0$-bounded SIN group $H$ is continuous.
\item \label{it:target:IACPolish} Every homomorphism $G \to H$ to a Polish SIN group $H$ is continuous.
\end{enumerate}
\end{Theorem}
One can of course add similar statements replacing the above classes of SIN target groups by any property in between, such as being $\aleph_0$-bounded / separable / second countable and additionally being Hausdorff or even metrizable. 
\begin{proof}
As any separable topological group is $\aleph_0$-bounded and since Polish groups are separable, it suffices to prove that \ref{it:target:IACPolish} implies \ref{it:target:IACAleph0}. 
Assume that \ref{it:target:IACPolish} holds and consider an $\aleph_0$-bounded SIN group $H$ and a homomorphism $\varphi: G \to H$. Consider the quotient group $\widetilde{H} = H/\cl{\sete{1_H}}$ and $\widetilde{\varphi}: G \to \widetilde{H}: g \mapsto \varphi(g)\cl{\sete{1_H}}$. Since $\widetilde{H}$ is $\aleph_0$-bounded Hausdorff SIN, by the modified Guran Theorem \ref{th:GuranSIN} there exist second countable metrizable SIN groups $\sete{K_i}_i$ such that $\widetilde{H}$ can be identified with a subgroup of $\prod_i K_i$. Consider $\widetilde{\varphi}: G \to \prod_i K_i$ and its component homomorphisms $\widetilde{\varphi}_i: G \to K_i$. Using \Cref{th:separableMetrisableEmbedsPolish} we can embed every $K_i$ into a Polish SIN group $\cl{K_i}$. By \ref{it:target:IACPolish} the associated homomorphisms $\widetilde{\varphi}_i: G \to \cl{K_i}$ are continuous. Hence so are the $\widetilde{\varphi}_i: G \to K_i$. It follows that $\widetilde{\varphi}: G \to \prod_i K_i$ is continuous as well and therefore the same is true for $\widetilde{\varphi}: G \to \widetilde{H}$. Finally \Cref{th:continuousTargetQuotientByClosureOfIdentity} yields that our original homomorphism $\varphi: G \to H$ is continuous, as required.
\end{proof}

We now provide a characterization of invariant automatic continuity in the spirit of \cite[Proposition 1.2]{Rosendal}, which characterizes automatic continuity in terms of continuity of homomorphisms into the universal Polish group $\mathrm{Homeo}([0,1]^{\mathbb{N}})$. 
The result below follows from the recent work of Doucha, see \cite{Doucha}, who proved the existence of a universal Polish SIN group $H$, meaning that any Polish SIN group is topologically isomorphic to a closed subgroup of $H$. 
\begin{Proposition}\label{prop:Universal}
	Let $G$ be a topological group. The following are equivalent.
	\begin{enumerate}[label=(\roman*)]
		\item \label{it:general} The topological group $G$ has (IAC).
		\item \label{it:special} Any homomorphism from $G$ to a fixed universal Polish SIN group is continuous. 
	\end{enumerate}
\end{Proposition}
\begin{proof}
	\ref{it:general} $\Rightarrow$ \ref{it:special}: This is a special case of the definition of (IAC).
	
	\ref{it:special} $\Rightarrow$ \ref{it:general}: Using \Cref{th:IACTarget} it suffices to show that any homomorphism $\varphi: G \to K$ into a Polish SIN group $K$ is continuous. By \cite[Theorem 0.2]{Doucha} there exists a universal Polish SIN group $H$, i.e.\ $K$ may be identified with a closed subgroup of $H$. As any homomorphism $G\to H$ is continuous by assumption, so is $\varphi:G\to K$.		
\end{proof}

One can also characterize (IAC) on a group $G$ by considering which group topologies $G$ admits. Below the comparison $\tau' \subseteq \tau$ of topologies means that $\tau$ is finer (or stronger) than $\tau'$. The corresponding statement for (AC), i.e.\ after removing the SIN condition everywhere in the theorem below, is also valid.
\begin{Theorem}
\label{th:IACTop}
Let $(G, \tau)$ be a topological group. Then the following are equivalent.
\begin{enumerate}[label=(\roman*)]
\item \label{it:top:IAC} The topological group $(G, \tau)$ has (IAC).
\item \label{it:top:finerThanSINBounded} $\tau' \subseteq \tau$ for every $\aleph_0$-bounded SIN group topology $\tau'$ on $G$.
\item \label{it:top:finerThanSINSC} $\tau' \subseteq \tau$ for every second countable SIN group topology $\tau'$ on $G$.
\end{enumerate}

Moreover, if $\tau$ is second countable Hausdorff SIN, then the following statement is also equivalent.
\begin{enumerate}[resume, label=(\roman*)]
\item \label{it:top:finerThanSINSCMetr} 
$\tau' \subseteq \tau$ for every second countable metrizable SIN group topology $\tau'$ on $G$.
\end{enumerate}
\end{Theorem}
As for \Cref{th:IACTarget} one may replace the presented classes of SIN topologies in the equivalences of \Cref{th:IACTop} by any of the properties $\aleph_0$-boundedness / separability / second countability and in the case where $\tau$ is second countable Hausdorff SIN, additionally add Hausdorffness or metrizability.
\begin{proof}
The implication \ref{it:top:finerThanSINBounded} $\Rightarrow$ \ref{it:top:finerThanSINSC} is obvious. We show that \ref{it:top:finerThanSINSC} $\Rightarrow$ \ref{it:top:IAC} $\Rightarrow$ \ref{it:top:finerThanSINBounded}.

\ref{it:top:finerThanSINSC} $\Rightarrow$ \ref{it:top:IAC}: Assume that $\tau_G \coloneqq \tau$ contains every second countable SIN group topology on $G$. By \Cref{th:IACTarget} it suffices to show that if $(H, \tau_H)$ is a second countable SIN group and $\varphi: (G, \tau_G) \to (H, \tau_H)$ a homomorphism, then $\varphi$ is continuous. Now the initial topology $\tau_\varphi$ from $\varphi$ is a second countable SIN group topology on $G$. By assumption it follows that $\tau_\varphi \subseteq \tau_G$ which means that $\varphi: (G, \tau_G) \to (H, \tau_H)$ is continuous.

\ref{it:top:IAC} $\Rightarrow$ \ref{it:top:finerThanSINBounded}: Assume that $(G, \tau_G)$ has (IAC) and that $\tau_{\aleph_0\SIN}$ is an $\aleph_0$-bounded SIN group topology on $G$. Then $\Id_G: (G, \tau_G) \to (G, \tau_{\aleph_0\SIN}): g \mapsto g$ must be continuous by \Cref{th:IACTarget}, implying that $\tau_{\aleph_0\SIN} \subseteq \tau_G$. 

This proves the first three equivalences. As the implication \ref{it:top:finerThanSINSC} $\Rightarrow$ \ref{it:top:finerThanSINSCMetr} is clear, it remains to show that \ref{it:top:finerThanSINSCMetr} $\Rightarrow$ \ref{it:top:IAC} when $\tau$ is second countable Hausdorff SIN. By \Cref{th:IACTarget} it suffices to prove that every group homomorphism $\varphi: G \to H$ to a second countable Hausdorff SIN group $H$ is continuous. Consider 
\[\varphi': G \to G \times H: g \mapsto (g, \varphi(g)).\]
Then $\varphi'$ is an injective homomorphism into the group $G \times H$, which is second countable Hausdorff SIN as a product of second countable Hausdorff SIN groups. Now the initial topology $\tau_{\varphi'}$ of $\varphi'$ on $G$ is again second countable Hausdorff SIN (it is Hausdorff as $\sete{1_G} = \varphi'^{-1}(\sete{1_{G \times H}})$ is $\tau_{\varphi'}$-closed by Hausdorffness of $G \times H$). By the Birkhoff-Kakutani Theorem \ref{th:BirkhoffKakutani} it is then even metrizable. Thus $\tau_{\varphi'} \subseteq \tau$ by assumption \ref{it:top:finerThanSINSCMetr}, so that $\varphi'$ is continuous (with respect to $\tau$). Therefore also its second component map $\varphi$ is continuous.
\end{proof}

It is not clear whether a similar statement involving Polish SIN group topologies would be equivalent as well. By the open mapping theorem for Polish groups this would lead to the statement that a Polish SIN group $(G, \tau)$ has (IAC) if (and only if) $\tau$ is the unique Polish SIN topology $G$ admits. In the setting of automatic continuity this is certainly false. For example, we can consider $\mathrm{PU}(n)$, which has a unique Polish group topology \cite[Theorem 11]{GartsidePejic}, but does not have automatic continuity \cite[Example 1.5]{RosendalSurvey}.

\section{Invariant automatic continuity of compact connected simple Lie groups}
\label{sec_cpt}
\subsection{Preliminaries on compact connected simple Lie groups}
We only consider finite dimensional Lie groups. A \textit{Lie group} is a group endowed with the structure of a (second countable Hausdorff) differentiable manifold such that the multiplication and inversion maps are smooth. A Lie group is called \textit{simple} if it is non-abelian and contains no non-trivial closed connected normal subgroups.

Fix a compact connected simple Lie group $L$. Then $\dim L > 0$, so that $L$ is uncountable. On the other hand $Z(L)$ is finite \cite[Proposition 23.3]{Bump}. A \textit{torus} of $L$ is a non-trivial compact connected abelian subgroup. One can show that $L$ admits maximal tori (with respect to inclusion) \cite[Theorem 5.4]{Sepanski} and that these are all conjugate \cite[Corollary 5.10(b)]{Sepanski}. The dimension of a maximal torus is called the \textit{rank} of $L$.
A celebrated result due to Cartan, see e.g.\ \cite[Theorem 16.4]{Bump}, is the following.
\begin{Theorem}[Cartan]
	\label{th:conjMaxTorus}
	Let $L$ be a compact connected Lie group and $T \subseteq L$ a maximal torus. Then every element of $L$ is conjugate to some element of $T$.
\end{Theorem}

In particular, it follows that $Z(L) \subseteq T$. 

\begin{Proposition}
	\label{th:compactLieBiInvMetric}
	Every compact Lie group admits a bi-invariant metric.
\end{Proposition}
\begin{proof}
	Let $L$ be a compact Lie group. Then $L$ is Hausdorff and first countable, hence metrizable by the Birkhoff-Kakutani Theorem \ref{th:BirkhoffKakutani}. Being compact, $L$ is also SIN. It then follows from \Cref{th:BirkhoffKakutaniSIN} that $L$ admits a bi-invariant metric.
\end{proof}

Let $L$ be a compact connected simple Lie group and $T$ a maximal torus. Then in \cite[p.~17]{StolzThom} the authors use (continuous) characters to construct a continuous map $\lambda: T \to [0, 1]$ such that $\lambda(t) = 0$ for some $t \in T$ if and only if $t \in Z(L)$ \cite[Proposition 4.1]{StolzThom}. Moreover, $\lambda$ can be used to measure the length of products of conjugates in the following sense, see \cite[Theorem 4.6]{StolzThom}.
\begin{Theorem}[Stolz-Thom]
	\label{th:lambdaProp}
	Let $L$ be a compact connected simple Lie group of rank $r$ and let $T$ be a maximal torus. Further let $s, t \in T$ with $\lambda(t) > 0$. For any even integer $m \in \mathbb{N} \setminus \sete{0}$ such that $\lambda(s) \leq m \lambda(t)$ one has
	\[s \in \left(t^L \cup t^{-L}\right)^{4mr^2}.\] 
\end{Theorem}
\noindent Here $t^L$ denotes the conjugacy class of $t$ and $t^{-L}$ that of $t^{-1}$.

\subsection{Compact connected simple Lie groups have (IAC)}
Equipped with the results of the previous sections, we can now prove that compact connected simple Lie groups have invariant automatic continuity. In the proof we show that such a group is invariant Steinhaus and then must have (IAC) by Theorem \ref{th:invSteinhaus}.

\begin{Theorem}\label{th:IACcptsimple}
Every compact connected simple Lie group of rank $r$ has the invariant Steinhaus property with exponent $8r^2$, hence has the invariant automatic continuity property (IAC).
\end{Theorem}
\begin{proof}
We denote the compact connected simple Lie group by $L$ and we fix a maximal torus $T$. Consider an arbitrary symmetric countably syndetic conjugacy-invariant subset $W \subseteq L$. Note that $W$ must be uncountable as $L$ is uncountable and hence in particular $W\not\subseteq Z(L)$. 
Take some $w \in W \setminus Z(L)$. By \Cref{th:conjMaxTorus} there exists some $t \in w^L \cap T$. Since $w \notin Z(L)$ we have $t \notin Z(L)$. As $W$ is conjugacy-invariant, we find that $t \in (W \cap T) \setminus Z(L)$. Then $\delta \coloneqq 2\lambda(t) > 0$.

By continuity $U \coloneqq \lambda^{-1}([0, \delta))$ is an open set in $T$ containing $1_T$. Let $u \in U$. Then $\lambda(u) < \delta = 2\lambda(t)$. \Cref{th:lambdaProp} implies that $u \in \left(t^L \cup t^{-L}\right)^{8r^2}$. Thus $U \subseteq \left(t^L \cup t^{-L}\right)^{8r^2}$. Since $W \ni t$ is symmetric, we have $t^{-1} \in W$ and conjugacy-invariance of $W$ implies that $t^L, t^{-L} \subseteq W$. Thus $U \subseteq W^{8r^2}$.

Using \Cref{th:compactLieBiInvMetric} we can find a bi-invariant metric $d: L \times L \to \mathbb{R}_{\geq 0}$ compatible with the topology on $L$. Then $d$ restricts to a compatible bi-invariant metric on $T$ and as $U \ni 1_T$ is open in $T$ there exists an open ball
\[B_T \coloneqq B_T(1_T, \eps) = \setp{s \in T}{d(1_T, s) < \eps} \subseteq U\]
for some $\eps \in \mathbb{R}_{>0}$. Now consider
\[B_L \coloneqq B_L(1_L, \eps) = \setp{g \in L}{d(1_L, g) < \eps}\]
and let $g \in B_L$. By \Cref{th:conjMaxTorus} there exists an $h \in L$ such that $s \coloneqq hgh^{-1} \in T$. Note that
\[d(1_T, s) = d(1_L, s) = d(1_L, hgh^{-1}) = d(h, hg) = d(1_L, g) < \eps\]
by bi-invariance of $d$. It follows that $s \in B_T \subseteq U \subseteq W^{8r^2}$. But as $W^{8r^2}$ is conjugacy-invariant, this means that $g \in W^{8r^2}$. We conclude that $B_L \subseteq W^{8r^2}$. As $B_L \ni 1_L$ is open in $L$, this shows that $L$ is invariant Steinhaus with exponent $8r^2$. Hence $L$ has (IAC) by Theorem \ref{th:invSteinhaus}.
\end{proof}

\section*{Acknowledgements}
The results in this article are mainly taken from the master's thesis of the first named author under supervision of the second named author. L.D. was partially supported by European Research Council Consolidator Grant 614195 RIGIDITY. P.A.D. was also supported by European Research Council Consolidator Grant 614195 RIGIDITY and by European Research Council Consolidator Grant 681207.


\begin{thebibliography}{9}

\bibitem[AT08]{ArhangelskiiTkachenko} Alexander Arhangel'skii and Mikhail Tkachenko, \emph{Topological groups and related structures}, Atlantic Press, 2008.




\bibitem[Bir36]{Birkhoff} Garett Birkhoff, \emph{A note on topological groups}, Compositio Mathematicae 3 (1936), 427-430.


\bibitem[Bum04]{Bump} Daniel Bump, \emph{Lie groups}, Graduate Texts in Mathematics, Springer, 2004.



\bibitem[Dou17]{Doucha}
Michal Doucha, 
\textit{Metrical universality for groups},
Forum Math. 29 (2017), no. 4, 847-872.

\bibitem[DT15]{DowerkThom} Philip A. Dowerk and Andreas Thom, \emph{Bounded normal generation and invariant automatic continuity}, August 12, 2015, \url{https://arxiv.org/pdf/1506.08549.pdf}. 




\bibitem[Gao09]{Gao} Su Gao, \emph{Invariant descriptive set theory}, CRC Press, Taylor and Francis Group, 2009.

\bibitem[GP08]{GartsidePejic} Paul Gartside and Bojana Peji\'c, \emph{Uniqueness of Polish group topology}, Topology and its Applications, Volume 155, Issue 9, 15 April 2008, 992-999.

\bibitem[Gur81]{Guran}
Ihor Guran,
\emph{On topological groups close to being Lindel\"of},
Soviet Math. Dokl. 23, 173-175 (1981).

\bibitem[Kak36]{Kakutani} Shizuo Kakutani, \emph{\"Uber die Metrisation der topologischen Gruppen}, Proceedings of the Imperial Academy, Volume 12, Number 4 (1936), 82-84.

\bibitem[Kal00]{Kallman}
{R. R. Kallman},
\textit{Every reasonably sized matrix group is a subgroup of $\mathrm{S}_{\infty}$};
Fund. Math. 164 (2000), 35-40.






\bibitem[LMW17]{LeMaitreWesolek}
Fran\c{c}ois Le Ma\^itre and and Phillip Wesolek,
\emph{On strongly just infinite profinite branch groups}, 
J. Group Theory 20 (2017), no. 1, 1-32. 





\bibitem[RZ10]{RibesZalesskii} Luis Ribes and Pavel Zalesskii, \emph{Profinite groups}, Second Edition, A Series of Modern Surveys in Mathematics, Springer, 2010.


\bibitem[Ros08]{Rosendal} Christian Rosendal, \emph{Automatic continuity in homeomorphism groups of compact 2-manifolds}, Israel J. Math. 166 (2008), 349-367.

\bibitem[Ros09]{RosendalSurvey} Christian Rosendal, \emph{Automatic continuity of group homomorphisms}, Bulletin of Symbolic Logic 15 (2):184-214 (2009). 

\bibitem[RS07]{RosendalSolecki} Christian Rosendal and S\l avomir Solecki, \emph{Automatic continuity of homomorphisms and fixed points on metric compacta}, Israel J. Math. 162 (2007), 349-371. 


\bibitem[Sep07]{Sepanski} Mark R. Sepanski, \emph{Compact Lie groups}, Graduate Texts in Mathematics, Springer, 2007.

\bibitem[ST13]{StolzThom} Abel Stolz and Andreas Thom, \emph{On the lattice of normal subgroups in ultraproducts of compact simple groups}, Proc. London Math. Soc. (3) 108 (2014), no.1, 73-102. 



\end{thebibliography}
\end{document}